\documentclass[10pt,a4paper]{amsart}

\usepackage{latexsym, amsthm}

\usepackage{amsfonts, amsmath, amssymb}

\usepackage{euscript,mathrsfs}

\newtheorem{theorem}{Theorem}[section]
\newtheorem{lemma}[theorem]{Lemma}
\newtheorem{proposition}[theorem]{Proposition}

\newtheorem{corollary}[theorem]{Corollary}

\theoremstyle{definition}
\newtheorem{definition}[theorem]{Definition}

\newtheorem{question}[theorem]{Question}
\newtheorem{notation}[theorem]{Notation}

\theoremstyle{remark}
\newtheorem{remark}[theorem]{Remark}
\def\Fq{{\mathbb F}_q}

\def\e{{\mathbf e}}

\def\FF{{\mathbb F}}

\def\Sd{{S_{\le d}}}

\def\hp3{\widehat{\mathbb P}^3}
\newcommand{\supp}{\mathrm{Supp}}
\newcommand{\RM}{\mathrm{RM}}

\def\ZZ{{\mathbb Z}}
\newcommand\hf[1]{{^a\mathrm{HF}_{#1}}}
\newcommand\hfs[1]{{^a\mathrm{HF}_{#1}(u)}}

\def\LT{\mathrm{LT}}

\def\ku{{\FF[x_1, \dots, x_m]_{\le u}}}

\begin{document}
\title[Generalized Hamming weights of affine cartesian codes]{Generalized Hamming weights of affine cartesian codes}
\author{Peter Beelen and Mrinmoy Datta}
\address{Department of Applied Mathematics and Computer Science, \newline \indent
Technical University of Denmark, DK 2800, Kgs. Lyngby, Denmark}
\email{pabe@dtu.dk, mrinmoy.dat@gmail.com}

\begin{abstract}
In this article, we give the answer to the following question: 
Given a field $\FF$, finite subsets $A_1,\dots,A_m$ of $\FF$, and $r$ linearly independent polynomials $f_1,\dots,f_r \in \FF[x_1,\dots,x_m]$ of total degree at most $d$. 
What is the maximal number of common zeros $f_1,\dots,f_r$ can have in $A_1 \times \cdots \times A_m$? 
For $\FF=\Fq$, the finite field with $q$ elements, answering this question is equivalent to determining the generalized Hamming weights of the so-called affine Cartesian codes. 
Seen in this light, our work is a generalization of the work of Heijnen--Pellikaan for Reed--Muller codes to the significantly larger class of affine Cartesian codes.  
\end{abstract}

\maketitle

\section{Introduction}

Let $\FF$ be a field and $A_1, \dots, A_m$ be finite non-empty subsets of $\FF$ consisting of $d_1, \dots, d_m$ elements respectively. For each $i = 1, \dots, m$ we write $A_i = \{\gamma_{i, 1}, \dots, \gamma_{i, d_i} \}$. 
We consider the finite subset of $\FF^m$ given by the cartesian product $\mathcal{A} = A_1 \times \cdots \times A_m.$ Without loss of generality we may assume that $d_1 \le \dots \le d_m$. Define $n := |\mathcal A| = d_1 \cdots d_m$.

Let $S:=\FF[x_1, \dots, x_m]$ denote the polynomial ring in $m$ variables $x_1, \dots, x_m$ and for an integer $d \ge 1$, denote by $\Sd (\mathcal{A})$ the vector subspace of $S$ consisting of polynomials $f$ with $\deg f \le d$ and $\deg_{x_i} f < d_i$ for $i=1, \dots, m$. 
In this article we give the answer to the following main question:
\begin{question}\label{q:main}
For a positive integer $r \le \dim \Sd (\mathcal{A})$, let $f_1, \dots, f_r$ be linearly independent elements of $\Sd (\mathcal{A})$. What is the maximum number of common zeroes that $f_1, \dots, f_r$ can have in $\mathcal{A}$?
\end{question}

Denoting by $Z(f_1,\dots,f_r)$ the set of common zeros of $f_1, \dots, f_r$ in $\FF^m$, we can reformulate this question as: What is the maximum cardinality of $Z(f_1,\dots,f_r) \cap \mathcal A$? As noted in \cite[Thm.3.1]{LRV}, 
we may assume that $d \le \sum_{i=1}^m (d_i - 1)$, since $x_1^{d_1-1}\cdots x_m^{d_m-1}$ is the monomial of highest possible degree in $\Sd (\mathcal{A})$. We will use the notation $k:=\sum_{i=1}^m (d_i - 1)$ in the remainder of this article.

Partial answers to Question \ref{q:main} are known, but the general case is still open. First of all, in case $r = 1$ it was answered in \cite[Prop.3.6]{LRV}. 
Furthermore, for  $A_1=\dots =A_m=\Fq$, the finite field with $q$ elements, the question was settled in \cite{HP} for all values of $r$ using, among others, the theory of order domains applied to Reed--Muller codes. 
In \cite{HP}, Question \ref{q:main} was answered in a reformulated form in terms of so-called generalized Hamming weights of certain error-correcting codes. 
Also in \cite {LRV} it was observed that the answer to Question \ref{q:main} for the case $r=1$ gives the minimum distance of what they called affine cartesian codes. It was brought to our attention by Olav Geil that, these codes were already studied in \cite{GT1} in a more general setting and the answer to Question \ref{q:main} for $r=1$ is a special case of \cite[Prop. 5]{GT1}.  Therefore, after having answered Question \ref{q:main}, we compute the generalized Hamming weights of affine cartesian codes. Moreover, we explicitly determine the duals of affine cartesian codes and as a consequence  obtain these weights for the duals as well. 

The article is organized as follows: In Section 2, we collect some results from the theory of affine Hilbert functions and their relations to counting the number of points on a zero dimensional affine variety. In Section 3, we revisit a combinatorial result of Wei \cite[Lemma 6]{W} and prove it completely in a more general setting. Next, in Section 4, we answer Question \ref{q:main} and in Section 5, we determine the generalized Hamming weights of affine cartesian codes and their duals.

\section{Affine Hilbert functions and number of points on a zero dimensional affine variety}

The set of common zeroes of $f_1, \dots, f_r$ in $\mathcal{A}$ is of course a finite subset of $\FF^m$. Therefore, it has a natural interpretation as a zero dimensional affine variety. For this reason, we explore in this section the theory of affine Hilbert functions and discuss its relation with the number of points on zero dimensional affine varieties. This relation will be used in subsequent sections. Many results on affine Hilbert functions exist in the literature. For a detailed discussion on the results mentioned in this section, one may for example refer to \cite{CLO} and \cite{NW}.

Let $\ku$ denote the subset of $S = \FF[x_1, \dots, x_m]$ consisting of polynomials of degree at most $u$. For an ideal $I$ of $S$, we denote by $I_{\le u}$ the subset of $I$ consisting of polynomials of degree at most $u$. Note that both $\ku$ and $I_{\le u}$ are vector spaces over $\FF$. The function
$$\hf{I} : \ZZ \to \ZZ \ \ \ \ \mathrm{given \ by} \ \ \ \ \hf{I} (u) = \dim \ku - \dim I_{\le u}$$
is called the \emph{affine Hilbert function} of $I$. One may readily observe that, if $I \subset J$ then $\hf{I}(u) \ge \hf{J} (u)$.

Similarly, given a subset $X$ of $\FF^m$ we define the \emph{affine Hilbert function} of $X$, denoted by, $\hfs{X}$ as $\hfs{X} = \hf{I(X)} (u)$, where $I(X)$ is the ideal of $S$ consisting of polynomials of $S$ that vanishes at every point of $X$. It is easy to show that, if $X \subset Y$ then $\hfs{X} \le \hfs{Y}$.

To compute the affine Hilbert function of a given ideal $I \subset S$, one can use the theory of monomial ideals, i.e., ideals generated by monomials. 
For a given graded order $\prec$ on $S$ one defines $\LT(I)$ to be the ideal generated by $\{\LT(f) : f \in I\},$ where $\LT(f)$ denotes the leading monomial of $f$ under $\prec$. 
Then we have the following well-known proposition. For a proof one may refer to \hbox{Section 3 of Chapter 9 of \cite{CLO}}.

\begin{proposition}\label{ahf}
Let $\prec$ be a graded order on $S$.
\begin{enumerate}
\item[(a)] For any ideal $I$ of $S$, we have $\hf{\LT (I)}(u)=\hf{I}(u)$ for any $u \in \ZZ$.

\item[(b)] If $I$ is a monomial ideal of $S$ then $\hf{I} (u)$ is the number of monomials of degree at most $u$ that do not lie in $I$.
\end{enumerate}
\end{proposition}
The next known proposition, taken from \cite[Lemma 2.1]{NW}, relates affine Hilbert functions of zero-dimensional ideals with the number of points in the corresponding variety. Similar statements (though formulated in the language of so-called footprints) can be found in \cite[Cor.2.5]{CLO2} and \cite[Cor.4.5]{G}.
\begin{proposition}\label{ZW}
Let $Y \subset \FF^m$ be a finite set. Then $|Y| = \hfs{Y}$ for all sufficiently large values of $u$.
\end{proposition}

Now we come back to Question \ref{q:main}. Note that $I(\mathcal{A})$ contains the polynomials
$$g_i := \prod_{j=1}^{d_i} (x_i - \gamma_{i, j}) \ \ \ \mathrm{for} \ i = 1, \dots, m.$$ This means that $\LT (I(\mathcal{A}))$ contains the monomials $x_1^{d_1}, \dots, x_m^{d_m}$. Further, given $r$ linearly independent polynomials $f \in \Sd (\mathcal{A})$
the ideal $$\mathcal{J} := \LT (I (Z(f_1, \dots, f_r) \cap \mathcal{A}))$$ contains the monomials $\LT(f_1), \dots, \LT(f_r)$ along with the monomials $x_1^{d_1}, \dots, x_m^{d_m}$. We may assume w.l.o.g.~that $\LT(f_1), \dots, \LT(f_r)$ are distinct using our assumption that $f_1, \dots, f_r$ are linearly independent. Thus,
\begin{equation}\label{eq:IJ}
\mathcal{I} := \langle \LT(f_1), \dots, \LT(f_r), x_1^{d_1}, \dots, x_m^{d_m} \rangle \subseteq \mathcal{J}.
 \end{equation}
This implies that $\hf{\mathcal{J}} (u) \le \hf{\mathcal{I}}(u)$ for all $u \in \ZZ$. By Propositions \ref{ahf} and \ref{ZW}, we have,
\begin{equation}\label{ub}
|Z(f_1, \dots, f_r) \cap \mathcal{A}| = \hf{\mathcal{J}} (u) \le \hf{\mathcal{I}}(u)
\end{equation}
for all sufficiently large values of $u$.

The above shows that Hilbert functions of monomial ideals can be used to answer Question \ref{q:main}. The following proposition gives a very useful way of determining such Hilbert functions, see \cite[\S 2.4, Lemma 2 and \S 9.3, Prop.3]{CLO} for a proof.

\begin{proposition}[\cite{CLO}]\label{compute} Let $M$ be a monomial ideal and $u$ be a positive integer.
\begin{enumerate}
\item[(a)]  Then $\hf{M}(u)$ is given by the number of monomials of degree at most $u$ that do not belong to $M$.
\item[(b)]  Let $M$ be generated by monomials $m_1, \dots, m_s$ and let $m$ be an arbitrary monomial. Then $m \in M$ if and only if $m_i | m$ for some $i =1, \dots, s$.
\end{enumerate}
\end{proposition}

\begin{corollary}\label{fp:ideal}
Let $M$ be a monomial ideal and $u$ be a positive integer. Then $\hf{M}(u)$ is given by the number of monomials of degree at most $u$ which are not divisible by any of the generators of $M$.
\end{corollary}

\begin{proof}
Immediately follows from Proposition \ref{compute}.
\end{proof}

Now we return to the ideal $\mathcal{I}$ defined in equation \eqref{eq:IJ}. We write $\mathcal{I} = \mathcal{I}_1 + \mathcal{I}_2$, where $\mathcal{I}_1 := \langle x_1^{d_1}, \dots, x_m^{d_m} \rangle$ and $\mathcal{I}_2 := \langle \LT(f_1), \dots, \LT(f_r) \rangle$. By Proposition \ref{compute}, any monomial that does not belong to $\mathcal{I}_1$ will be of the form $x_1^{a_1} \cdots x_m^{a_m}$ where $a_i \le d_i - 1$ for all $i = 1, \dots, m$.  Consequently, the monomials that do not belong to $\mathcal{I}_1$ are, naturally, in one-to-one correspondence with points in $F = [0, d_1 - 1] \times \cdots \times [0, d_m -1]$. In more concrete terms, if $\mathcal{M}_{I_1}$ denotes the set of all monomials that do not belong to $\mathcal{I}_1$, then the map $\phi : \mathcal{M}_{I_1} \to F$ given by $x_1^{a_1} \cdots x_m^{a_m} \mapsto (a_1, \dots, a_m)$ gives such a bijection. In particular, if we assume that $u \ge k$ then the monomials of degree at most $u$ in $\mathcal{M}_{I_1}$ are in one-to-one correspondence with elements of $F$.
Further, if $m,n$ are two monomials in $\mathcal M(I_1)$ we have that $m|n$ if and only if $\phi(m) \le_P \phi(n)$.

In light of Proposition \ref{compute}, a monomial $m \in \mathcal{M}_{I_1}$, is in $\mathcal{I}_2$ if and only if $\phi (m_i) \le_P \phi (m)$ for some $i = 1, \dots, r$. Here $\le_P$ denotes the natural partial ordering on $F$, defined by $$(i_1,\dots,i_m) \le_P (j_1,\dots,j_m) \ \makebox{if and only if} \ i_1 \le j_1, \dots, i_m \le j_m.$$
This leads us to consider the so-called shadow of a collection of elements in $F$:
\begin{definition}
Let $u_1,\dots,u_r \in F$, then we define the shadow of $u_1,\dots,u_r$ in $F$ as
 $$\Delta(u_1,\dots,u_r):=\{u \in F : u_i \le_P u \ \makebox{ for some } \ i=1,\dots,r\}.$$
\end{definition}

Combining the above discussion and Proposition \ref{compute}, we have following:
\begin{equation}\label{hilfunc}
\hf{\mathcal{I}}(u) = |F \setminus \Delta (\phi (\LT(f_1)), \dots, \phi (\LT(f_r)))|, \ \ \ \mathrm{where} \ \ u \ge k.
\end{equation}
Hence, from equations \eqref{ub} and \eqref{hilfunc}, we get that
\begin{equation}\label{numfp}
|Z(f_1, \dots, f_r) \cap \mathcal{A}| \le |F \setminus \Delta (\phi (\LT(f_1)), \dots, \phi (\LT(f_r)))|.
\end{equation}
and hence
\begin{equation}\label{maxfp}
|Z(f_1, \dots, f_r) \cap \mathcal{A}| \le \max \{|F \setminus \Delta (a_1, \dots, a_r)| : a_1,\dots,a_r \in F_{\le d} \},
\end{equation}
where $F_{\le d} := \{(i_1, \dots, i_m) \in F : i_1 + \dots + i_m \le d\}$. Note that for $r=1$, inequality \eqref{numfp} is given in \cite[Cor.13]{GT}.

\section{Generalization of a combinatorial theorem by Wei}

Inequality \eqref{maxfp} gives a way to investigate Question \ref{q:main} using purely combinatorial means. What is needed is to determine the minimum cardinality of the shadow $\Delta (b_1, \dots, b_r)$ given $r$ distinct elements $b_1,\dots,b_r \in F_{\le d}$. 
In this section we will determine this minimum cardinality. Our approach is to generalize \cite{W}, where the case $A_1=\cdots = A_m=\FF_2$ was settled. 
It should be noted that we actually found an error in the proof of \cite[Lemma 6]{W}. 
This has some impact, since \cite[Lemma 6]{W} also was used in \cite{HP} to deal with the case $A_1=\cdots = A_m=\Fq$. 
Fortunately, the material in this section (notably Theorem \ref{wei}) implies that Lemma 6 in \cite{W} is correct and thus fully justifies its use in \cite{HP}. 
  
For the convenience of the reader let us recap the notation we have used so far as well as introduce some further notation that we will use in this section.

\begin{notation} 
{\ }
\begin{enumerate}
\item[(a)] Let $d_1 \le \dots \le d_m$ be integers and $k = (d_1 - 1) + \dots + (d_m -1)$.
\item[(b)] $F := \{0,\dots, d_1 - 1\} \times \dots \times \{0, \dots, d_m - 1\}$.
\item[(c)] For $a=(i_1,\dots,i_m) \in F$, define $\deg(a):=i_1+\cdots+i_m$. 
\item[(c)] For $u \le k$, define $F_u := \{a \in F : \deg(a) = u\}$ and $F_{\le u} = \{a \in F : \deg(a) \le u\}$.
\item[(d)] Let $S \subset F_u$ and $|S| = r$. Denote by $L(S)$ the set of first $r$ elements of $F_u$ in descending lexicographic order.
\item[(e)] $(i_1,\dots,i_m) \le_P (j_1,\dots,j_m) \ \makebox{if and only if} \ i_1 \le j_1, \dots, i_m \le j_m.$
\item[(f)] Let $u \le v \le k$ and $S \subseteq \Fq$.  Define $\Delta (S) = \{ a \in F : \exists b \in S, b \le_P a \}$ and $\Delta_v (S) = \Delta (S) \cap F_v$.
\item[(g)] For $S \subset F_{\le u}$ with $|S| = r$, denote by $M(S)$ the first $r$ elements of $F_{\le u}$ in descending lexicographic order.
\end{enumerate}
\end{notation}

Like in \cite{HP}, the following theorem due to Clements and Lindstr\"{o}m, will be an essential combinatorial tool. 

\begin{theorem}[Cor.1 \cite{CL}]\label{CLT}
For $u \le k$ let $S \subseteq F_u$. Then $\Delta_{u+1} (L(S)) \subseteq L (\Delta_{u+1} (S))$.
\end{theorem}

\begin{corollary}\label{cl}
For $u \le v \le k$ and $S \subset F_u$ we have $\Delta_v(L(S)) \subseteq \L (\Delta_v(S))$. In particular, $|\Delta_v (L(S))| \le |\Delta_v(S)|$.
\end{corollary}

\begin{proof}
For $u=v$ there is nothing to prove. The case $v = u+1$ follows from Theorem \ref{CLT}. If $v=u+2$ then we have,
\begin{align*}
\Delta_{u+2} (L(S)) &= \Delta_{u+2} (\Delta_{u+1} (L(S))) \\
&\subseteq \Delta_{u+2} (L (\Delta_{u+1} (S))) \subseteq L (\Delta_{u+2} (\Delta_{u+1} (S))) = L (\Delta_{u+2} (S)).
\end{align*}
The rest of the proof follows by induction on $v-u$.
\end{proof}

\begin{corollary}\label{ccl}
For $u \le k$ and $S \subset F_u$, we have $|\Delta (L(S))| \le |\Delta (S)|$.
\end{corollary}

\begin{proof}
Note that $|\Delta (S)| = \displaystyle{\sum_{v=u}^k} |\Delta_v (S)| \ge \displaystyle{\sum_{v=u}^k} |\Delta_v (L(S))| = |\Delta (L(S))|$. The inequality follows from Corollary \ref{cl}.
\end{proof}

\begin{lemma}\label{proplex}
Let $1 \le v \le k$ and write $u = v-1$. Choose $y \in F_v$ and consider $a:= \max_{lex} \{f \in F_u : f \le_{lex} y\}$. Then $a \le_P y$.
\end{lemma}

\begin{proof}
Write $y - a = (0, \dots, 0, c_i, c_{i+1}, \dots, c_m)$, where $1 \le i \le m$ and $c_i \in \ZZ$. Since $a \le_{lex} y$, we have $c_i \ge 0$. We divide the proof in two cases.

\textbf{Case 1:} If $c_j \ge 0$ for all $j > i$ then we have $a  \le_P y$ and we are done.

\textbf{Case 2:} There exists $\ell > i$ such that $c_{\ell} < 0$. Choose $j := \min \{\ell : c_{\ell} < 0\}$.  Note that, by definition of $j$, we have $c_{\ell} \ge 0$ for all $i \le \ell < j$.

\textit{Subcase $1$:} Suppose $c_i >1$.  Let $\tilde{a}:= a + \e_i - \e_j$, where $\e_{s}$ denotes the $m$ tuple with $1$ in the $s$-th coordinate and zeroes elsewhere.  It follows trivially that $a <_{lex} \tilde{a}$. Moreover, the first nonzero coordinate in $y - \tilde{a}$ is $c_i - 1$ which is positive. This implies that $\tilde{a} <_{lex} y$ and hence
\begin{equation}\label{eq:max}
a <_{lex} \tilde{a} <_{lex} y.
\end{equation}
Moreover, $\deg \tilde{a} = u$. To see that $\tilde{a} \in F$, we observe that
\begin{enumerate}
\item $0 \le a_i \le a_i + 1 = y_i - c_i + 1 < y_i \le d_i - 1$ since $c_i > 1$ and
\item $d_j - 1\ge a_j \ge a_j - 1 = y_j - c_j -1 \ge y_j \ge 0$ since $c_j < 0$.
\end{enumerate}
Hence, $\tilde{a} \in F_u$ contradicting the maximality of $a$.

\textit{Subcase $2$:} Assume that $c_i =1$ and $c_{\ell} > 0$ for some $i < \ell < j$. It is easy to see that the same $\tilde{a}$ as in subcase 1 satisfies the inequality \eqref{eq:max} which again violates the maximality of $a$.

\textit{Subcase $3$:} Assume that $c_i =1$ and $c_{\ell} = 0$ for all $i < \ell < j$. Since $c_j < 0$ and $c_i = 1$ with $\sum_{\ell = i}^m c_{\ell} = v - u = 1$, there exists $n > j$ such that $c_n >0$. Let $\tilde{a} = y - \e_n$. Clearly $\tilde{a} \le_{lex} y$. Also, the first nonzero coordinate of $\tilde{a} - a$ is $c_i = 1 > 0$. Thus $\tilde{a}$ satisfies the inequality \eqref{eq:max} and clearly $\tilde{a} \in F_u$ since $c_n > 0$ which contradicts the maximality of $a$.
\end{proof}

\begin{remark}\label{iter}
Note that, one could derive the conclusion of Lemma \ref{proplex} for any $0 <u < v \le k$ by applying the lemma iteratively.
\end{remark}

\begin{lemma}\label{claim}
Assume that $1 \le u \le v \le k$. Let $M(r)$ denote the first $r$ elements of $F_{\le v}$ in descending lexicographic order. Let $M_u = M(r) \cap F_u$ and $r_u := |M_u|$.
Then  $\Delta_v (M_u) \subseteq M_v \subseteq \Delta_v (M_u^*) $, where $M_u^*$ consists of the first $r_u+1$ elements of $F_u$ in descending lexicographic order for $1 \le u \le v$.
\end{lemma}

\begin{proof}
If $y \in \Delta_v (M_u)$, then there exists $i \in M_u$ and $a \in F_{v-u}$ such that $y = i+a$. Hence, $y \ge_P i$ which implies $y \ge_{lex} i$. The fact $i \in M(r)$ thus implies that $y \in M(r)$. Also, $y \in F_v$. Thus, $y \in M(r) \cap F_v = M_v$. This proves the first inclusion.

Let $y \in M_v$. Define $a:= \max_{lex} \{f \in F_u : f \le_{lex} y\}$. Remark \ref{iter} implies that $a \le_P y$. If $a \in M_u$ then we are done. So we may assume that $a \not\in M_u$. Note that $M_u$ consists of the first $r_u$ elements of $F_u$ in descending lexicographic order. This implies that $a \le_{lex} f_{r_u + 1}$, where $f_1, \dots, f_{r_u + 1}$ denote the first $r_u + 1$ elements of $F_u$ in descending lexicographic order. If $a = f_{r_u + 1}$ then  $a \in M_u^*$ and hence $y \in \Delta_v (M_u^*)$. Now suppose, if possible, that $a <_{lex} f_{r_u + 1}$. By maximality of $a$ we have $y <_{lex} f_{r_u + 1}$. However, since $y \in M(r)$, we have $f_{r_u + 1} \in M(r)$ (by definition of $M(r)$) which implies that $f_{r_u + 1} \in M_u$. This is a contradiction since $|M_u| = r_u$. This completes the proof.
\end{proof}

\begin{lemma}\label{claim2}
For $r \ge 1$ and $v \le k$, let $M(r)$ denote the set of first $r$ elements of $F_{\le v}$ in descending lexicographic order. For $u \le v$, we write $M_u = M(r) \cap F_u$. Then $|\Delta (M(r))| = r - |M_v| + |\Delta (M_v)|$.
\end{lemma}

\begin{proof}
It follows from Lemma \ref{claim} that,
\begin{equation}\label{bigcup}
\bigcup_{u < v} \Delta_v (M_u) \subset M_v.
\end{equation}
Now,
$$|\Delta(M(r))| = |\Delta (M(r)) \cap F_{<v}| +  |\Delta (M(r)) \cap F_{\ge v}| =  |\Delta (M(r) \setminus M_v) \cap F_{<v}| + |\Delta (M_v)|.$$
Note that, $M(r) \setminus M_v$ consists of first $r - |M_v|$ elements of $F_{\le v-1}$ in descending lexicographic order. Hence,
$\Delta_{v-1} (M(r) \setminus M(v)) = M_{v-1}$ by applying \eqref{bigcup} to $M(r) \setminus M_v$ (on the $(v-1)$-th level). Reasoning iteratively,
$\Delta_{u} (M(r) \setminus M(v)) = M_{u}$ for all $u \le v-1$. This proves the lemma.
\end{proof}

The following theorem is a generalization of \cite[Lemma 6]{W} and our proof approach is similar as in Wei's paper. 
However, as mentioned before, Wei's somewhat terse proof contains a mistake which is why we have chosen to give a fully detailed proof of Theorem \ref{wei}.

\begin{theorem}\label{wei}
For $v \le k$, let $S \subseteq F_{\le v}$  with $|S| = r$. Then $|\Delta  (M (r))| \le |\Delta (S)|$, where, as before, $M(r)$ denotes the first $r$ elements of $F_{\le v}$ in descending lexicographic order.
\end{theorem}

\begin{proof}
For $u \le v$, define $S_u = S \cap F_u$. We divide the proof into two cases:

\textbf{Case 1:}  Suppose that $|S_v| = r_v + \alpha \ge r_v$ for some $\alpha \ge 0$. Then,
$$|\Delta (S_v)| \ge |\Delta (M_v)| + \alpha.$$
This follows from Corollary \ref{ccl} applied on a subset of $S_v$ consisting of $r_v$ elements and the contribution of shadows in $F_v$ of the remaining $\alpha$ elements of $S \setminus S_v$. Further,
\begin{align*}
|\Delta (S)| &= |\Delta_{<v} (S)| + |\Delta_{\ge v}(S)| \\
&\ge |\Delta_{<v} (S)| + |\Delta (S_v)|  \ \ \ \  \mathrm{(follows \ from \  Lemma \ \ref{claim})} \\
&\ge |S \cap F_{<v}| + |\Delta (S_v)| = r - |S_v| + |\Delta (S_v)|
\end{align*}

Hence,
\begin{align*}
|\Delta (S)| &\ge r - |S_v| + |\Delta (S_v)| \\
&\ge r - r_v - \alpha + |\Delta (M_v)| + \alpha \\
&= |\Delta (M(r))| \ \ \ \ \text{(from Lemma \ref{claim2})}
\end{align*}

\textbf{Case 2:}
Now suppose that $|S_v|<|M_v|$. Since $|S|=|M(r)|$, this implies that there exists $u<v$ such that $|S_u|>|M_u|$. Hence, $|M_u^*| \le |S_u|$. By Lemma \ref{claim} and Theorem \ref{cl} we have $|M_v| \le |\Delta_v(M_u^*)| \le |\Delta_v(S_u)|.$ Hence,
\begin{align*}
|\Delta(S)| &\ge r-|S_v|+|\Delta_{\ge v}(S_u)| \\
&> r - |M_v|+|\Delta_{\ge v}(S_u)| \\
&= r - |M_v|+|\Delta(\Delta_v(S_u))| \\
&\ge r - |M_v|+|\Delta(M_v)|=|\Delta(M(r))|
\end{align*}
The last equality follows from Lemma \ref{claim2}.
\end{proof}

\section{Answer to Question \ref{q:main}}

In this section we give the answer to Question \ref{q:main} in Theorem \ref{thm:mt}. There are two main steps in the proof of this theorem. 
First, the combinatorial theory developed in the previous section is used to obtain an upper bound for $|Z(f_1,\dots,f_r) \cap \mathcal A|$. 
Further we construct an explicit family of $\FF$-linearly independent polynomials in $\Sd(\mathcal A)$  that attains this upper bound.
Our results are more general than the results presented in \cite{HP}, but 
some of the ideas are akin to that in \cite[Section 5]{HP}. For instances, the
Lemmas \ref{rth} and \ref{fp} are direct generalizations of Lemma 5.8 and Proposition 5.9 in \cite{HP} and Definition \ref{def:fa} is similar to Definition 4.10 from \cite{HP}.

\begin{lemma}\label{rth}
Let $0 \le r \le d_1 \cdots d_m - 1$ be an integer.
Write $$d_1 \cdots d_m - r = \sum_{i=1}^m a_{r, i} \prod_{j=i+1}^m d_j.$$
Then $(a_{r, 1}, \dots, a_{r, m})$ is the $r$-th tuple of $F$ in descending lexicographic order.
\end{lemma}

\begin{proof}
Define a map
$$\phi : F \to \{0, 1, \dots, d_1 \cdots d_m - 1\} \ \ \ \mathrm{given \ by} \ \ \ \  (a_1, \dots, a_m) \mapsto \sum_{i=1}^m a_i \prod_{j=i+1}^m d_j.$$
Since any integer $0 \le r < d_1 \dots d_m$ can be expressed uniquely as $r = \sum_{i=1}^m a_i \prod_{j=i+1}^m d_j$, it follows that $\phi$ is surjective. Moreover, $$(a_1, \dots, a_m) \le_{lex} (b_1, \dots, b_m) \iff \sum_{i=1}^m a_i \prod_{j=i+1}^m d_j \le \sum_{i=1}^m b_i \prod_{j=i+1}^m d_j.$$
The claim follows from the fact that $d_1 \dots d_m - r$ is the $r$-th highest element of $\{0, 1, \dots, d_1 \cdots d_m - 1\}$.
\end{proof}

\begin{lemma}\label{fp}
Let $a_1, \dots, a_r$ be the first $r$ elements of $F_{\le d}$ in descending lexicographic order. Then,
$$\Delta (a_1, \dots, a_r) = \{a \in F : a_r \le_{lex} a\}.$$
Moreover, if $a_r = (a_{r, 1}, \dots, a_{r, m})$ then 
$$|\Delta (a_1, \dots, a_r)| = d_1 \cdots d_m  - \displaystyle{\sum_{i=1}^m a_{r, i} \prod_{j=i+1}^m d_j}.$$
\end{lemma}

\begin{proof}
Let $b \in \Delta (a_1, \dots, a_r)$. There exists $i \le r$ such that $b \ge_P a_i$. Thus, $b \ge_{lex} a_i \ge_{lex} a_r$. Consequently,
$\Delta (a_1, \dots, a_r) \subseteq \{ b \in F : a_r \le_{lex} b \}.$
Conversely, let $b \ge_{lex} a_r$. If $b = a_r$ then there is nothing to prove. So we may assume that $b >_{lex} a_r$. Writing $b=(b_1,\dots,b_m)$, there exists $t \le m$ such that $b_i = a_{r, i}$ for $i < t$ and $b_t > a_{r,t}$. Define an $m$-tuple
$c = (c_1, \dots, c_m)$ as follows:
$$c_i =
\begin{cases}
a_i  \ \ \mathrm{if} \ \ i<t \\
a_t + 1 \ \mathrm{if} \ \ i = t \\
0 \ \ \ \mathrm{otherwise}
\end{cases}
$$
Clearly, $b \ge_P c$ and $c \ge_{lex} a_r$. Note that $\deg c \le \deg a_r + 1$. If $\deg c = \deg a_r + 1$, then $a_r = (a_{r, 1}, \dots, a_{r, t}, 0, \dots, 0)$ which implies that $c \ge_P a_r$. On the other hand, if $\deg c \le \deg a_r \le d$ then $c = a_i$ for some $i \le r$. Consequently, $b \in \Delta (a_1, \dots, a_r)$. This shows that $b \in \Delta (a_1, \dots, a_r)$. The claim about the number of elements follows from the Lemma \ref{rth}.
\end{proof}

\begin{proposition}\label{mt}
Let $f_1, \dots, f_d \in \Sd (\mathcal{A})$ be linearly independent over $\Fq$. Then
$$|Z(f_1, \dots, f_r) \cap \mathcal{A}| \le \sum_{i=1}^m a_{r,i} \prod_{j=i+1}^m d_j,$$
where $(a_{r,1}, \dots, a_{r, m})$ is the $r$-th element of $F_{\le d}$ in descending lexicographic order.
\end{proposition}

\begin{proof}
This follows from inequality \eqref{maxfp}, Theorem \ref{wei} and Lemma \ref{fp}.
\end{proof}

We now construct a family of $\FF$-linearly independent polynomials $f_1,\dots,f_r$ in $\Sd (\mathcal{A})$ such that the cardinality of $Z(f_1,\dots,f_r) \cap \mathcal A$ attains the upper bound obtained in Proposition \ref{mt}. Recall that,  $A_i=\{\gamma_{i,1},\dots,\gamma_{i,d_i}\}$ for $i=1, \dots, m$.

\begin{definition}\label{def:fa}
For \hbox{$b = (b_1, \dots, b_m) \in F_{\le d}$} define the polynomial,
$$f_{b} = \prod_{s=1}^m \prod_{t=1}^{b_s}(x_s - \gamma_{s, t}).$$
\end{definition}

\begin{proposition}\label{maximal}
Let $a_1, \dots, a_r$ be the first $r$ elements of $F_{\le d}$ in descending lexicographic order.
Then, $$|Z(f_{a_1}, \dots, f_{a_r}) \cap \mathcal{A}| = \sum_{i=1}^m a_{r,i} \prod_{j=i+1}^m d_j,$$ where
$a_r = (a_{r,1}, \dots, a_{r, m})$.
\end{proposition}

\begin{proof}
Note that the map $\psi : \mathcal{A} \to F$ defined by $(\gamma_{1, i_1}, \dots, \gamma_{m, i_m}) \mapsto (i_1 - 1, \dots, i_m - 1)$ is a bijection.

For $b \in F_{\le d}$ we see that $f_{b} (\gamma_1, \dots, \gamma_m) \neq 0$ if and only if \hbox{$\gamma_i \in \{\gamma_{i, t} | t > a_s\}$} for all $i = 1, \dots, m$. 
This implies that $f_{b} (\gamma_1, \dots, \gamma_m) \neq 0$ if and only if $\psi  (\gamma_1, \dots, \gamma_m)  \in \Delta (b)$. Consequently, for $a_1, \dots, a_r \in F_{\le d}$, we have $\gamma \in \mathcal{A} \setminus Z(f_{a_1}, \dots, f_{a_r})$ if and only if $\psi (\gamma) \in \Delta (a_1, \dots, a_r)$. 
Thus, $|\mathcal{A} \setminus Z(f_{a_1}, \dots, f_{a_r})| = |\Delta (a_1, \dots, a_r)|$.  
In particular, if $a_1, \dots, a_r$ are the first $r$ elements of $F_{\le d}$ in descending lexicographic order we see that, $|\mathcal{A} \setminus Z(f_{a_1}, \dots, f_{a_r})| = |\Delta (a_1, \dots, a_r)| =  d_1 \cdots d_m  - \displaystyle{\sum_{i=1}^m a_{r, i} \prod_{j=i+1}^m d_j}$, where the last equality follows from Lemma \ref{fp}. Thus, $|Z(f_{a_1}, \dots, f_{a_r}) \cap \mathcal{A}| = \sum_{i=1}^m a_{r,i} \prod_{j=i+1}^m d_j,$
\end{proof}

\begin{theorem}\label{thm:mt}
We have
$$\max |Z(f_1, \dots, f_r) \cap \mathcal{A}| = \sum_{i=1}^m a_{r,i} \prod_{j=i+1}^m d_j,$$
where $(a_{r,1}, \dots, a_{r, m})$ is the $r$-th element of $F_{\le d}$ in descending lexicographic order and where the maximum is taken over all $\FF$-linearly independent $f_1, \dots, f_r \in \Sd (\mathcal{A})$.
\end{theorem}

\begin{proof}
This follows from Proposition \ref{mt} and Proposition \ref{maximal}.
\end{proof}

\section{Affine cartesian codes and their higher weights}
In this section we relate our results with coding theory and obtain a complete determination of the generalized Hamming weights of a class of codes containing the well-known Reed--Muller codes as a particular case. Throughout this section we assume $\FF = \Fq$, where $\Fq$ denotes the finite field with $q$ elements, but otherwise we use the same notation as before. In particular, we assume that $d_1 \le \dots \le d_m$ are positive integers and $A_1, \dots, A_m$ are subsets of $\Fq$ of cardinality $d_1, \dots, d_m$ respectively. As before, denote by $\mathcal{A}$ the cartesian product $\mathcal{A}=A_1 \times \cdots \times A_m$. Also we fix an enumeration $P_1, \dots, P_n$ of elements in $\mathcal{A}$ and a positive integer $d \le k:=\sum_{i=1}^m (d_i - 1).$

Recall that, a linear code of length $N$ and dimension $K$ is simply a linear subspace of $\Fq^N$ of dimension $K$. One class of codes related to the setting in this article is obtained as follows.
\begin{definition}
Let $ev$ be the map defined by
$$ev: S_{\le k} (\mathcal{A}) \to \Fq^{|\mathcal A|} \ \ \ \mathrm{by} \ \ \ \ f \mapsto (f(P_1), \dots, f(P_n)).$$
Then for $d \le k$ we define $AC_q (d, \mathcal A):=ev(\Sd (\mathcal{A}))$. 
\end{definition}

Note that $AC_q (d, \mathcal A)$ is a linear code since $ev$ is a linear map. It has length $n:=|\mathcal A|=d_1\cdots d_m$ and it follows from the injectivity of $ev$ that the dimension of $AC_q (d, \mathcal{A})$ is  $\dim \Sd(\mathcal A)$.
The codes obtained in this way are called \textit{affine cartesian codes}.
Affine cartesian codes were defined in \cite{LRV} and further studied in, for example, \cite{C, KK, GT, CN}. In \cite[Theorem 3.8]{LRV} the authors determined the minimum distance of these codes. 
\begin{remark}\label{rem}
\
\begin{enumerate}
\item[(a)] If $A_1 = \dots = A_m = \Fq$, then the code $AC_q (d, \mathcal A)$ is the generalized Reed-Muller code $RM_q (d, m)$.
\item[(b)] If $A_1 = \dots = A_m = \Fq\backslash\{0\}$, the code $AC_q (d, \mathcal A)$ is a toric code, see \cite{H}.
\end{enumerate}
\end{remark}
In this section we completely determine the generalized Hamming weights of affine cartesian codes. 
For the ease of the reader, we recall the definition of generalized Hamming weights of linear codes.
\begin{definition}
Let $D \subseteq \Fq^N$ be a subspace of dimension $r$. The \textit{support} of $D$ is defined to be
$$\supp (D) := \{i : \mathrm{there \ exists} \ (c_1, \dots, c_N) \in D \ \mathrm{such \ that} \ c_i \neq 0\}.$$
Let $C \subset \Fq^N$ be a code (i.e., linear subspace) of dimension $K$. For $1 \le r \le K$, the $r$th \textit{generalized Hamming weight} of $C$, denoted by $d_r (C)$ is defined as,
$$d_r (C) := \min \{|\supp (D)| : D \subseteq C, \dim D = r\}.$$
The quantity $d_1(C)$ is simply the minimum distance of the code $C$. 
\end{definition}

\begin{theorem}\label{hwacc}
Let $d_r (d, \mathcal A)$ denote the $r$-th generalized Hamming weight of $AC_q (d, \mathcal A)$. Then
$$d_r (d, \mathcal A) = 1 + \sum_{i=1}^m a_{r,i} \prod_{j=i+1}^m d_j,$$
where $(a_{r,1}, \dots, a_{r, m})$ is the $r$-th element of $F_{\ge k -  d}$ in ascending lexicographic order.
\end{theorem}

\begin{proof}
It follows from Theorem \ref{thm:mt} that
\begin{equation}\label{conn1}
d_r (d, \mathcal A) = d_1 \cdots d_m - \sum_{i=1}^m b_{r,i} \prod_{j=i+1}^m d_j,
\end{equation}
where $(b_{r,1}, \dots, b_{r, m})$ is the $r$-th element of $F_{\le d}$ in descending lexicographic order. Further we note that,
\begin{equation}\label{conn2}
d_1 \cdots d_m - 1 = \sum_{i=1}^m (d_i - 1) \prod_{j=i+1}^m d_j.
\end{equation}
From \eqref{conn1} and \eqref{conn2} we see that,
\begin{align*}
d_r (d,\mathcal A) &= 1 + (d_1 \cdots d_m - 1) - \sum_{i=1}^m b_{r,i} \prod_{j=i+1}^m d_j \\
&= 1 + \sum_{i=1}^m (d_i - 1 - b_{r, i}) \prod_{j=i+1}^m d_j.
\end{align*}
We define $a_{r, i} = d_i - 1 - b_{r, i}$ for $i =1, \dots, m$. The assertion of the theorem now follows noting that the map $(b_1, \dots, b_m) \mapsto (d_1 - 1 - b_1, \dots, d_m - 1 - b_m) : F_{\le d} \to F_{\ge k -  d}$ is a bijection that reverses the lexicographic order on elements of $F_{\le d}$.
\end{proof}
This theorem has a number of corollaries, relating it to previously known results. In the first place, we recover a result concerning the minimum distance of $AC_q (d, \mathcal{A})$.
\begin{corollary}\cite[Theorem 3.8]{LRV}
The minimum distance of $AC_q (d, \mathcal A)$ is given by $(d_{j+1} - \ell) d_{j+2} \cdots d_m$, where $j \ge 0$ and $\ell \le d_{j+1} - 1$ are uniquely determined integers such that $d = \ell+\sum_{i=1}^j (d_i - 1)$.
\end{corollary}
\begin{proof}
The first element of $F_{\le d}$ in descending lexicographic order is given by $(d_1 - 1, \dots, d_j -1, \ell, 0, \dots, 0)$. 
From Theorem \ref{hwacc} and its proof we see that the minimum distance of $AC_q (d, \mathcal A)$ is equal to
\begin{align*}
& 1 + \sum_{i=1}^j (d_i - 1 - (d_i -1)) \prod_{s=i+1}^m d_s + (d_{j+1} - 1 - \ell) d_{j+2} \cdots d_m + \sum_{i=j+2}^m (d_i - 1) \prod_{s=i+1}^m d_s \\
&= (d_{j+1} - \ell) d_{j+2} \cdots d_m.
\end{align*}
The last equality follows noting that
$$1 - d_{j+2} \cdots d_m  + \sum_{i=j+2}^m (d_i - 1) \prod_{s=i+1}^m d_s = 0,$$
which completes the proof.
\end{proof}
As another consequence of Theorem \ref{thm:mt}, we recover the generalized Hamming weights of Reed-Muller codes $RM_q(d,m)$. This was obtained in \cite[Theorem 7]{W} for $q=2$ and in \cite[Theorem 5.10]{HP} for any prime powers $q$.

\begin{corollary}\cite[Thm. 5.10]{HP}
The $r$-th higher weights of $\RM_q (d, m)$ is given by
$$d_r (d, m) = 1 + \sum_{i=1}^m a_i q^{m-i}$$
where $(a_1, \dots, a_m)$ denotes the $r$-th element of $F_{\ge m(q-1) -  d}$ in ascending lexicographic order.
\end{corollary}

\begin{proof}
This follows directly from Remark \ref{rem}(a) and Theorem \ref{thm:mt}.
\end{proof}

We finish this section by an observation on the generalized Hamming weights of $AC_q(d,\mathcal A)$. Generalized Hamming weights have a number of properties. One of these, called Wei-duality \cite[Thm.3]{W} is the following. For a linear code $C \subset \Fq^N$ of dimension $K$, consider the linear subspace of $\Fq^N$
$$C^\perp:=\{(e_1,\dots,e_N) : c_1\cdot e_1+\cdots + c_N\cdot e_N=0  \ \mathrm{for \ all} \ (c_1, \dots, c_N) \in C\}.$$
The code $C^\perp$ has dimension $N-K$ and is called the dual code of $C$. 
The following well-known statement, that relates the higher weights of a linear code to that of its dual, is sometimes referred to as Wei duality and can be found in \cite[Thm 3]{W}.
\begin{equation}\label{eq:weidual}
\{d_1(C),\dots,d_K(C)\} \cup \{N+1-d_1(C^\perp),\dots,N+1-d_{N-K}(C^\perp)\} = \{1,\dots,N\}.
\end{equation}
Note that the union in equation \eqref{eq:weidual} is a disjoint union. 
Further, a direct computation using equation \eqref{conn2} shows that the sets $\{d_r(d, \mathcal A) : 1 \le r \le \dim \Sd(\mathcal A)\}$ and $\{n+1-d_r(k-d-1, \mathcal A) : 1 \le r \le \dim S_{\le k-d-1}(\mathcal A)\}$ are disjoint and have union $\{1,\dots,n\}$.  
For Reed-Muller codes, this is very simple to derive from Wei duality, since $RM_q(d,m)^\perp=RM_q(m(q-1)-d-1,m)$. 
But for affine cartesian codes this is not true in general. However,  the following result offers an explanation of the above observation. 
As in Section 2, we will use the polynomials $g_i:=\prod_{j=1}^{d_i}(x_i-\gamma_{i,j}).$ 
Further, denote by $g_i'(x_i)$, the partial derivative of $g_i$ with respect to $x_i$. We use our enumeration $P_1,\dots,P_n$ of the elements of $\mathcal A$ as in the beginning of this Section.
\begin{theorem}\label{thm:duals}
We have $$AC_q(d,\mathcal A)^\perp:=\{(w_1c_1,\dots,w_{n}c_{n}) : (c_1,\dots,c_{n}) \in AC_q(k-d-1,\mathcal A)\},$$ where $w_j^{-1}=\prod_{i=1}^m g_i'(P_j).$
\end{theorem}
\begin{proof}
Let us for convenience write $$C:=\{(w_1d_1,\dots,w_{n}c_{n}) : (c_1,\dots,c_{n}) \in AC_q(k-d-1,\mathcal A)\}.$$ 
Observe that $$\dim C = \dim AC_q(k-d-1,\mathcal A) = \dim S_{\le k-d-1}(\mathcal A)=\dim S_{\le k}(\mathcal A) - \dim S_{\ge k-d}(\mathcal A),$$
while
$$\dim AC_q(d,\mathcal A)^\perp =|\mathcal A| - \dim AC_q(d,\mathcal A)= \dim S_{\le k}(\mathcal A) - \dim S_{\le d}(\mathcal A).$$
Since $\dim S_{\le d}(\mathcal A)=\dim S_{\ge k-d}(\mathcal A)$, both codes $C$ and $AC_q(d,\mathcal A)^\perp$ have the same dimension. Therefore, the theorem follows once we show that $C \subset AC_q(d,\mathcal A)^\perp.$

Now let $0 \le \ell \le d_i-1$. Since the univariate polynomials 
$\sum_{j=1}^{d_i} \frac{g_i(x_i)}{x_i-\gamma_{i,j}}\frac{\gamma_{i,j}^\ell}{g_i'(\gamma_{i,j})}$ and $x_i^\ell$ have the same evaluation for any element of $A_i$, but have degree strictly less than $d_i$, they are equal. 
Comparing coefficients of $x_i^{d_i-1}$, we obtain that
\begin{equation}\label{eq:m=1}
\sum_{j=1}^{d_i}\frac{\gamma_{i,j}^\ell}{g_i'(\gamma_{i,j})}=\left\{ \begin{array}{rl}1 & \makebox{ if $\ell=d-1$}\\ 0 & \makebox{otherwise.}\end{array}\right.
\end{equation} 
Any codeword $c\in C$ is of the form $c=(w_1 f(P_1),\dots,w_{n}f(P_{n})) \in C$ for some $f \in S_{\le k-d-1}(\mathcal A)$ 
and likewise any codeword $\tilde{c}\in AC_q(d,\mathcal A)$ is of the form $\tilde{c}(\varphi(P_1),\dots,\varphi(P_n))$ for some $\varphi \in \Sd (\mathcal A).$ 
To show that $c_1\cdot \tilde{c}_1+\cdots+c_n\cdot \tilde{c}_n=0$, it is enough to show that this equality holds whenever $f$ and $\varphi$ are monomials.
Therefore, we will assume that $f$ and $\varphi$ are monomials from now on and write $f \cdot \varphi = x_1^{e_1} \cdots x_m^{e_m}$.
Since $\deg f \cdot \varphi \le k-d-1+d=k-1$, there exists at least one value of $i$ such that $e_i=\deg_{x_i} f\cdot \varphi < d_i-1$. 
Then we have 
$$\sum_{j=1}^n c_j\tilde{c}_j=\sum_{j=1}^n \frac{x_1^{e_1}\cdots x_m^{e_m} (P_j)}{(g_1'(x_1)\cdots g_m'(x_m))(P_j)}=\prod_{i=1}^m \sum_{j=1}^{d_i}\frac{\gamma_{i,j}^{e_i}}{g_i'(\gamma_{i,j})}.$$
Equation \eqref{eq:m=1} and the fact that $e_i<d_i-1$ for at least one value of $i$, then imply that $\sum_{j=1}^n c_j\tilde{c}_j=0$. This shows that $c \in AC_q(d,\mathcal A)^\perp$, which completes the proof.  
\end{proof}
\begin{corollary}
The $r$-th generalized Hamming weight of the code $AC_q(d,\mathcal A)^\perp$ is given by $d_{r}(k-d-1,\mathcal A)$.
\end{corollary}
\begin{proof}
Theorem \ref{thm:duals} directly implies that the codes $AC_q(d,\mathcal A)^\perp$ and $AC_q(k-d-1,\mathcal A)$ have the same generalized Hamming weights.  
\end{proof}

\section{Acknowledgments}

The authors would like to gratefully acknowledge the following foundations and institutions: 
Peter Beelen is supported by The Danish Council for Independent Research (Grant No. DFF--4002-00367). Mrinmoy Datta is supported by The Danish Council for Independent Research (Grant No. DFF--6108-00362).

\end{document}